\documentclass[12pt]{amsart}
\usepackage{amscd, amsfonts, amssymb}
\usepackage{fullpage}
\usepackage{latexsym}
\usepackage{verbatim}
\usepackage{enumerate}
\usepackage[colorlinks=true]{hyperref}

\newcommand\cc{\mathfrak c}
\renewcommand\gg{\mathfrak g}

\newcommand\kk{\mathfrak k}
\newcommand\mm{\mathfrak m}

\newcommand{\inprod}[1]{\langle #1\rangle}

\newcommand\inverse{{^{-1}}}
\renewcommand{\check}{^{\vee}}

\newcommand\ra{\rightarrow}

\newcommand{\tuple}[1]{{\mathbf {#1}}}

\DeclareMathOperator{\Ad}{Ad}

\DeclareMathOperator{\rk}{rk}

\DeclareMathOperator{\GL}{GL}

\DeclareMathOperator{\SL}{SL}

\DeclareMathOperator{\PGL}{PGL}

\DeclareMathOperator{\Lie}{Lie}

\DeclareMathOperator{\sat}{sat}

\DeclareMathOperator{\End}{End}

\DeclareMathOperator{\Mat}{Mat}
\DeclareMathOperator{\im}{im}

\numberwithin{equation}{section}

\newtheorem{thm}[equation]{Theorem}

\newtheorem{lem}[equation]{Lemma}
\newtheorem{cor}[equation]{Corollary}
\newtheorem{prop}[equation]{Proposition}

\theoremstyle{definition}
\newtheorem{defn}[equation]{Definition}
\newtheorem{exmp}[equation]{Example}

\theoremstyle{remark}
\newtheorem{rem}[equation]{Remark}
\theoremstyle{remark}


\title[$G$-complete reducibility and semisimple modules]
{$G$-complete reducibility and semisimple modules}

\author[M.\  Bate]{Michael Bate}
\address
{Department of Mathematics,
University of York,
York YO10 5DD,
United Kingdom}
\email{michael.bate@york.ac.uk}

\author[S. Herpel]{Sebastian Herpel}
\address
{Fakult\"at f\"ur Mathematik,
Ruhr-Universit\"at Bochum,
D-44780 Bochum, Germany}
\email{sebastian.herpel@rub.de}

\author[B.\ Martin]{Benjamin Martin}
\address
{Mathematics and Statistics Department,
University of Canterbury,
Private Bag 4800,
Christchurch 8140,
New Zealand}
\email{B.Martin@math.canterbury.ac.nz}

\author[G. R\"ohrle]{Gerhard R\"ohrle}
\address
{Fakult\"at f\"ur Mathematik,
Ruhr-Universit\"at Bochum,
D-44780 Bochum, Germany}
\email{gerhard.roehrle@rub.de}

\begin{document}

\begin{abstract}
Let $G$ be a connected reductive algebraic group defined 
over an algebraically closed field 
of characteristic $p > 0$. 
Our first aim in this note is 
to give concise and uniform proofs for two fundamental and deep results 
in the context of Serre's notion of $G$-complete reducibility, 
at the cost of less favourable bounds.
Here are some special cases of these results:
Suppose that the index $(H:H^\circ)$ is prime to $p$ and
that $p > 2\dim V-2$ for some faithful $G$-module $V$.
Then the following hold:
(i) $V$ is a semisimple $H$-module if and only if $H$ is 
$G$-completely reducible;
(ii) $H^\circ$ is reductive if and only if $H$ is $G$-completely reducible.

We also discuss two new related results:
(i)  if $p \ge \dim V$ for some $G$-module $V$ and
$H$ is a $G$-completely reducible subgroup of $G$, 
then $V$ is a semisimple $H$-module
-- this generalizes  Jantzen's  
semisimplicity theorem (which is the case $H = G$);
(ii)  if $H$ acts 
semisimply on $V \otimes V^*$ for some faithful $G$-module $V$, 
then $H$ is $G$-completely reducible. 
\end{abstract}

\maketitle


\section{Introduction}
\label{sec:intro}

Throughout, $G$ is a connected reductive linear algebraic group
defined over an algebraically closed field 
of characteristic $p > 0$ and 
$H$ is a closed subgroup of $G$.
Following  Serre \cite{serre2}, we say that 
$H$ is \emph{$G$-completely reducible} ($G$-cr for short) 
provided that whenever $H$ is contained in a parabolic subgroup $P$ of $G$,
it is contained in a Levi subgroup of $P$; 
for an overview of this concept see for instance \cite{serre1} and 
\cite{serre2}.
Note that in case $G = \GL(V)$ a subgroup $H$ is $G$-cr exactly when 
$V$ is a semisimple $H$-module.
Recall that if $H$ is $G$-cr, 
then the identity component $H^\circ$ of $H$ is reductive,  \cite[Property 4]{serre1}.

Let $V$ denote a rational $G$-module
and let $\rho : G \to \GL(V)$ be the representation of $G$
afforded by $V$.
Following Serre \cite{serre1}, 
we call $V$ \emph{non-degenerate} provided 
$({\rm ker} \rho)^\circ$ is a torus. 

First we consider two important and deep theorems in this context, 
\cite[Thm.\ 5.4]{serre2} and \cite[Thm.\ 4.4]{serre2}, 
which provide necessary and sufficient conditions for a subgroup 
$H$ of $G$ to be $G$-cr provided $p$ is sufficiently large.

\begin{thm}{\cite[Thm.\ 5.4]{serre2}}
\label{thm:nv}
Suppose that $p > n(V)$ for some rational $G$-module $V$.
\begin{itemize}
\item[(i)] 
If $H$ is $G$-completely reducible, then $V$ is a semisimple $H$-module. 
\item[(ii)]
Suppose that $V$ is non-degenerate.
If $V$ is semisimple as an $H$-module, then 
$H$ is $G$-completely reducible. 
\end{itemize}
\end{thm}

Here the invariant $n(V)$ is
defined as follows: let $T$ be a maximal torus of $G$ and let 
$\lambda$ be a $T$-weight of $V$. Define
$n(\lambda) = \sum_{\alpha>0}\inprod{\lambda,\alpha\check}$, 
where the sum is taken over all positive roots of $G$ 
with respect to $T$. Then define 
$n(V) = \sup \{n(\lambda)\}$, where the 
supremum is taken over all $T$-weights $\lambda$ of $V$,
 \cite[\S 5.2]{serre2}.
The proof of Theorem \ref{thm:nv} is elaborate and complicated;
it depends on 
the full force of the following result. 

\begin{thm}{\cite[Thm.\ 4.4]{serre2}}
\label{thm:ag}
Suppose that $p \ge a(G)$ and that 
$(H:H^\circ)$ is prime to $p$.
Then $H^\circ$ is reductive  if and only if 
$H$ is $G$-completely reducible.
\end{thm}

Here the invariant $a(G)$ is defined as follows:
for $G$ simple, set $a(G) = \rk(G) +1$, 
where $\rk(G)$ is the rank of $G$.
For $G$ reductive, let $a(G) = \sup(1, a(G_1), \ldots, a(G_r))$, 
where $G_1, \ldots, G_r$ are the simple components of $G$,
cf.\ \cite[\S 5.2]{serre2}.

We emphasize that Theorem \ref{thm:ag} is a consequence of a 
number of deep theorems
due to Jantzen \cite{jantzen0} (Theorem \ref{thm:jantzen}) 
and McNinch  \cite{mcninch2}
in case $G$ is classical 
and Liebeck and Seitz \cite{liebeckseitz0}
in case $G$ is of exceptional type, where
the latter involves complicated and long case-by-case analyses.
Given that the proofs of both these theorems are intricate,
it is desirable to have uniform arguments for them 
even under additional restrictions on $p$.
We present new concise and uniform proofs of these two results in
Theorems \ref{thm:ss-vs-cr} and \ref{thm:red-vs-cr}
with different bounds on $p$. 
Here we are particularly interested in obtaining 
short proofs for sufficient conditions for $G$-complete reducibility.
Unfortunately, though not unexpectedly, 
the brevity and uniformity 
do come at the expense of less favourable bounds on $p$.

In \cite{serre1}  and \cite{serre2}, Serre gave an alternative proof
of his tensor product theorem 
\cite[Thm.\ 1]{serre0} via the concept of $G$-cr subgroups. 
Theorems \ref{thm:nv} and \ref{thm:ag} are both part of this approach.
In Section \ref{sec:variations} we essentially argue the other way round:
based on a special case of 
the aforementioned 
tensor product theorem 
(Theorem \ref{thm:serre}), we first 
derive a short proof of Theorem \ref{thm:nv} (in Theorem \ref{thm:ss-vs-cr})
and in turn use part of that result 
to obtain a concise and uniform 
proof of Theorem \ref{thm:ag} (in Theorem \ref{thm:red-vs-cr}), 
with a worse bound on $p$.

\medskip

Our next result  generalizes 
Jantzen's semisimplicity Theorem \ref{thm:jantzen} to $G$-cr  subgroups
of $G$. 

\begin{thm}
\label{thm:jantzen-cr-subgps}
If  $p \ge \dim V$ and $H$ is $G$-completely reducible,
then $V$ is a semisimple $H$-module. 
\end{thm}

Theorem \ref{thm:jantzen-cr-subgps}  is also of interest, as the  bound $p > n(V)$ in Theorem \ref{thm:nv}(i) does not apply in case
$V$ admits a non-restricted composition factor,
cf.\ Remark \ref{rem:nv-vs-dim}.
The proof of  Theorem \ref{thm:jantzen-cr-subgps} requires 
the force of Theorem \ref{thm:nv}(i) and 
\cite[Thm.\ 1]{serre0}.

\medskip

Serre's notion of saturation in $\GL(V)$ 
(see Definition \ref{def:saturation}) is an important tool in the
theory of complete reducibility, see \cite{serre0} and \cite{serre1}.
As an application of Theorem \ref{thm:jantzen-cr-subgps} we show in Corollary \ref{cor:Hsat}
that if $p \ge \dim V$ and $H$ is a $G$-cr
subgroup of $G \le \GL(V)$, then 
the saturation of $H$ in  $\GL(V)$
is completely reducible in the saturation of $G$ in $\GL(V)$.

\medskip

Our final result is similar in spirit to Theorem \ref{thm:nv}(ii)
giving a sufficient semisimplicity condition for $H$ to be $G$-cr,
but strikingly it does not require any restriction on $p$.

\begin{thm} 
\label{thm:sufficient-cr}
If $H$ acts semisimply on $V \otimes V^*$ 
for some non-degenerate $G$-module $V$,
then $H$ is $G$-completely reducible and $\rho(H)$ is separable in $\rho(G)$.
\end{thm}

Recall that a subgroup $H$ of $G$ is said to be \emph{separable in $G$} if 
its scheme-theoretic centralizer is smooth, i.e.,\ if its
global and infinitesimal centralizers have the same dimension,
cf.\ \cite[Def.\ 3.27]{BMR}. 
In \cite{BMRT}, we study the interaction
between this notion of separability and the concept of $G$-complete
reducibility. 
Several general theorems concerning $G$-complete 
reducibility require some separability hypothesis, 
e.g.,\ see \cite[Thm.\ 3.35]{BMR}, \cite[Thm.\ 3.46]{BMR}.
In \cite[Thm.\ 1.2]{BMRT}, we show that any subgroup of $G$ is
separable in $G$ provided $p$ is very good for $G$.
Note that the special case of Theorem \ref{thm:sufficient-cr}
when $G = \GL(V)$  follows from  \cite[Thm.\ 3.46]{BMR},
since $\Lie(\GL(V)) \cong V \otimes V^*$.
The proof of Theorem \ref{thm:sufficient-cr} is based on a variant of
Richardson's tangent space argument
(cf.\ the separability statement of \cite[Thm.\ 1]{slodowy}), 
see Lemma \ref{lem:descend}.

\section{Preliminaries}
\label{sec:prelims}

We maintain the notation from the Introduction. 
In particular, $G$ is a connected reductive linear algebraic group
defined over an algebraically closed field 
of characteristic $p > 0$
and $H$ is a closed subgroup of $G$.
Moreover, $V$ is a rational $G$-module and 
$\rho : G \to \GL(V)$ is the representation of $G$
afforded by $V$.

First we recall Jantzen's 
fundamental semisimplicity result, \cite[Prop.\ 3.2]{jantzen0}.

\begin{thm}
\label{thm:jantzen}
If $p \ge \dim V$, then $V$ is semisimple.
\end{thm}

We continue with the special case 
for connected reductive groups
of Serre's seminal tensor product 
theorem, \cite[Prop.\ 8]{serre0}.

\begin{thm}
\label{thm:serre}
Suppose $p > 2\dim V - 2$.
If $V$ is semisimple, then so is $V\otimes V^*$.
\end{thm}

Both Theorems \ref{thm:jantzen} and \ref{thm:serre} 
have conceptional and uniform proofs
and both bounds are sharp (cf.\ \cite[Rem.(2), p.\ 260]{jantzen0} and \cite[\S 1.3]{serre0}). 

\medskip

Let $H$ be a closed
subgroup of $G$ such that $H^\circ$ is reductive.
We say that $(G,H)$ is a \emph{reductive pair} if
the Lie algebra $\Lie H$ 
of $H$ is an $H$-module direct
summand of the Lie algebra $\Lie G$ 
of $G$, cf.\ \cite{rich2}. 
Our next result is \cite[Cor.\ 3.36]{BMR}.

\begin{prop}
\label{prop:red-pair-gl-cr}
Suppose that $(\GL(V),G)$ is a reductive pair.
If  $V$ is a semisimple $H$-module, then
$H$ is $G$-completely reducible.
\end{prop}

Next we recall \cite[Cor.\ 2.13]{BMRT}. 

\begin{prop}
\label{prop:red-pair-gl-sep}
If $(\GL(V),G)$ is a reductive pair, then every subgroup of $G$ is separable in $G$.
\end{prop}

Suppose that $p \ge \dim V$, so that
every non-trivial unipotent element in $\GL(V)$ has order $p$.
We recall Serre's notion of \emph{saturation} in this instance,
cf.\ \cite{serre1}.
Let $u \in \GL(V)$ be unipotent. Then there is a nilpotent element $\epsilon \in \End(V)$
with $\epsilon^p = 0$ such that $ u = 1 + \epsilon$. 
For $t \in \mathbb G_a$ we can define $u^t$ by
$u^t = (1 + \epsilon)^t = 1 + t\epsilon + \binom{t}{2} \epsilon^2 + \cdots + \binom{t}{p-1} \epsilon^{p-1}$.
Then $\{u^t \mid t \in  {\mathbb G}_a\}$ is a closed connected subgroup  of $\GL(V)$
isomorphic to ${\mathbb G}_a$.

\begin{defn}
\label{def:saturation}
Suppose $p \ge \dim V$.
Let $H$ be a closed subgroup of $\GL(V)$.  
We say that $H$ is \emph{saturated} if
for each unipotent $u \in H$ we have 
$u^t \in H$ for all $t \in  \mathbb G_a$.
The \emph{saturated closure} $H^{\sat}$ of $H$ in $\GL(V)$ is the 
smallest saturated subgroup of $\GL(V)$ containing~$H$.
\end{defn}

There is a notion of saturation for any connected reductive group $G$,
but this is considerably 
more subtle, see \cite{serre1} for details.

The next result is the special case when $G = \GL(V)$ in  \cite[Thm.\ 5.3]{serre2}.
It follows since parabolic and Levi subgroups of $\GL(V)$ are saturated.

\begin{lem}
\label{lem:saturation}
Suppose that $p \ge \dim V$. 
Let $H$ be a closed subgroup of $\GL(V)$. 
Then $V$ is semisimple as an $H$-module if and only if it is semisimple as an $H^{\sat}$-module.
\end{lem}

The following is one of the key properties of saturated subgroups, 
\cite[Property 3]{serre1}.

\begin{lem}
\label{lem:saturation-index}
Suppose $H$ is a saturated subgroup of $\GL(V)$. 
Then 
$(H: H^\circ)$ is prime to $p$.
\end{lem}

\section{Variations on Theorems  \ref{thm:nv} and  \ref{thm:ag}}
\label{sec:variations}
We begin by showing that 
Theorems \ref{thm:jantzen} and \ref{thm:serre} and the 
bound on $p$ in the latter guarantee that 
$(\GL(V),\rho(G))$ is a reductive pair,
which is crucial for some of our subsequent arguments.

\begin{thm} 
\label{thm:red-pair}
Suppose that $p > 2\dim V-2$.
Then $(\GL(V),\rho(G))$ is a reductive pair.
\end{thm}

\begin{proof}
Since $p > 2\dim V-2$, we also have $p \ge \dim V$. 
Thus $V$ is semisimple, by
Theorem \ref{thm:jantzen}.
Thanks to Theorem \ref{thm:serre}, 
$V \otimes V^* \cong \Lie(\GL(V))$ is also semisimple.
Consequently, $\Lie \rho(G)$ is a direct $\rho(G)$-module 
summand of  $\Lie(\GL(V))$. 
\end{proof}

\begin{rem}
\label{rem:sharp-bound}
The bound in Theorem \ref{thm:red-pair} is sharp. 
For, let $p = 2$, let $G = \SL_2$, and let $V$ be the 
natural module for $G$.
Since $G$ is not separable in itself, Proposition \ref{prop:red-pair-gl-sep}
implies that $(\GL(V), G)$ is not a reductive pair.
In fact, although Theorem \ref{thm:red-pair} asserts that generically every 
representation $V$  of $G$ gives rise to a reductive pair $(\GL(V),\rho(G))$,
this is \emph{never} the case if $p$ is bad for $G$ and $V$ is non-degenerate, 
cf.\ Remark \ref{rem:separable}.
\end{rem}

The following is our variant of Theorem \ref{thm:nv}:

\begin{thm}
\label{thm:ss-vs-cr}
Let $H$ be a closed subgroup of $G$ and let $V$ be a $G$-module.
\begin{itemize}
\item[(i)] 
Suppose that $p \ge \dim V$ and that $(H : H^\circ)$ is prime to $p$.
If $H$ is $G$-completely reducible, then $V$ is a semisimple $H$-module. 
\item[(ii)]
Suppose that $V$ is non-degenerate and $p > 2\dim V-2$.
If $V$ is semisimple as an $H$-module, then 
$H$ is $G$-completely reducible. 
\end{itemize}
\end{thm}

\begin{proof}
First suppose as in (i). 
Since $H$ is $G$-cr, $H^\circ$ is reductive, by \cite[Property 4]{serre1}.
Since $p \ge \dim V$ and $(H : H^\circ)$ is prime to $p$, 
it follows from Theorem \ref{thm:jantzen}  (applied to $H^\circ$) and 
\cite[Lem.\ 3.1]{jantzen0} that $V$ is a semisimple $H$-module.

Next suppose as in (ii).
Since $(\GL(V),\rho(G))$ is a reductive pair, 
by Theorem \ref{thm:red-pair}, and $V$ is a semisimple $H$-module, 
it follows from  
Proposition \ref{prop:red-pair-gl-cr} that $\rho(H)$ is $\rho(G)$-cr.
Since $V$ is non-degenerate, \cite[Lem.\ 2.12(ii)(b)]{BMR} 
implies  that $H$ is $G$-cr.
\end{proof}

\begin{rem}
\label{rem:LieG}
Note that Theorem \ref{thm:nv}(ii) (and thus Theorem \ref{thm:ss-vs-cr}(ii))
holds  in particular cases for considerably weaker bounds.
For instance, in \cite[Thm.\ 1.7]{BMRT}, 
we showed that if $p$ is very good for $G$ and
$H$ acts semisimply on the adjoint module $\Lie G$, then $H$ is $G$-cr.
The reverse implication fails under this bound, 
cf.\ \cite[Rem.\ 3.43(iii)]{BMR}.
Here 
$n(\Lie G) = 2 h - 2$, where $h$ is the Coxeter number of $G$,
cf.\ \cite[Cor.\ 5.5]{serre2}.
\end{rem}

Here is our variation of Theorem \ref{thm:ag}:

\begin{thm}
\label{thm:red-vs-cr}
Suppose that $p > 2\dim V-2$ for a non-degenerate $G$-module $V$ and that 
$(H:H^\circ)$ is prime to $p$.
Then $H^\circ$ is reductive  if and only if  
$H$ is $G$-completely reducible.
\end{thm}

\begin{proof}
First suppose that $H^\circ$ is reductive.
Since $p > 2\dim V-2$, we also have $p \ge \dim V$. 
Thus $V$ is a semisimple $H^\circ$-module, by
Theorem \ref{thm:jantzen} (applied to $H^\circ$).
Moreover, since $(H:H^\circ)$ is prime to $p$, it follows from 
\cite[Lem.\ 3.1]{jantzen0}
that $V$ is a semisimple $H$-module.
The result now follows from Theorem \ref{thm:ss-vs-cr}(ii). 

The reverse implication is immediate, by 
\cite[Property 4]{serre1}. 
\end{proof}

Clearly, our bound in Theorem \ref{thm:red-vs-cr} 
is much worse than the bound $a(G)$ 
in Theorem~\ref{thm:ag}.

\begin{rem}
\label{rem:separable}
It follows from Theorem \ref{thm:red-pair} 
and Proposition \ref{prop:red-pair-gl-sep}
that only separable subgroups of 
$G$ are captured in Theorems \ref{thm:ss-vs-cr}(ii) and  \ref{thm:red-vs-cr}.
In \cite[Ex.\ 4.2]{herpel}, the second author showed that 
if $p$ is bad for $G$, then there always exists 
a non-separable subgroup of $G$
(likewise if $G$ is simple and $p$ is not very good for $G$). 
Consequently, in this case $(\GL(V),\rho(G))$ can't be a reductive pair 
for \emph{any} non-degenerate rational $G$-module $V$, 
by Proposition \ref{prop:red-pair-gl-sep}, cf.\ \cite[Rem.\ 4.3]{herpel}. 
Each of the non-separable subgroups constructed in \cite[Ex.\ 4.2]{herpel}
is a regular reductive subgroup of $G$ and hence is $G$-cr,
thanks to \cite[Prop.\ 3.20]{BMR}. 
\end{rem}

\section{Proof of Theorem \ref{thm:jantzen-cr-subgps}}

We now come to our generalization of 
Jantzen's semisimplicity  Theorem \ref{thm:jantzen} to $G$-cr subgroups
of $G$; Theorem \ref{thm:jantzen}
is the special  case of Theorem \ref{thm:jantzen-cr-subgps} when $H = G$.

\begin{proof}[Proof of Theorem \ref{thm:jantzen-cr-subgps}]
Suppose that $p\geq \dim V$ and $H$ is $G$-cr. 
By Theorem \ref{thm:jantzen}, $V$ is a semisimple $G$-module. 
So to show that $V$
is also semisimple as an $H$-module, we may assume that $V = L(\lambda)$ is simple
of highest weight $\lambda = \lambda_0 + p \lambda_1 + \dots + p^r \lambda_r$, with
restricted weights $\lambda_i$.
Set $L_i := L(p^i\lambda_i) \cong L(\lambda_i)^{[i]}$, 
the $i$th \emph{Frobenius twist} of $L(\lambda_i)$.
Then  
$V = L_0 \otimes L_1 \otimes \cdots \otimes L_r$, 
by Steinberg's tensor product theorem.
Since $p\geq \dim V$, we also have $p\geq \dim L_i = \dim L(\lambda_i)$ for each $i$, so that 
$p > n(\lambda_i) = n(L(\lambda_i))$ for each $i$, 
according to \cite[Lem.\ 1.2]{jantzen0}.
By Theorem \ref{thm:nv}(i), 
each $L(\lambda_i)$ is a  semisimple $H$-module and hence
so is each Frobenius twist $L_i$. 
Moreover, $p\geq \dim V$ also implies that $p > \sum_i (\dim L_i -1)$.
We therefore may apply Serre's tensor product theorem 
\cite[Thm.\ 1]{serre0} to deduce
that $V$ is a semisimple $H$-module.
\end{proof}

Note that Theorem \ref{thm:jantzen-cr-subgps} 
shows that Theorem \ref{thm:ss-vs-cr}(i) is valid even without the restriction on 
the index of $H^\circ$ in $H$.
However, the proof of Theorem \ref{thm:jantzen-cr-subgps}  requires the full force of Theorem  \ref{thm:nv}(i), whereas our proof of Theorem \ref{thm:ss-vs-cr}(i) does not, so this is of independent interest.

\begin{rem}
Proposition \ref{prop:red-pair-gl-cr} asserts that
under the assumption that $(\GL(V),G)$ is a reductive pair,
$H$ is $G$-cr provided $V$ is a semisimple $H$-module.
In Theorem \ref{thm:jantzen-cr-subgps}
we prove the reverse implication under the assumption that $p \ge \dim V$.

Even under the seemingly stronger condition 
that $(\GL(V),G)$ is a reductive pair and $V$ is a semisimple $G$-module, the statement of 
Theorem \ref{thm:jantzen-cr-subgps} is false without the restriction on $p$.
Such an example is already known 
thanks to a construction from unpublished work of Serre, 
cf.\ \cite[Ex.~4.7]{BMRT}.
We now give a different example:
Let $p=3$, $q=9$ and let $G= \SL_2$.  Set $H = G(q)$.
Clearly, $H$ is $G$-cr. 
The simple $G$-module $V = L(1+q+q^2)$  
is isomorphic to $L(1)\otimes L(1)^{[1]}\otimes L(1)^{[2]}$ ,
by Steinberg's tensor product theorem, 
where the superscripts denote $q$-twists.
Then $\dim V = 8 > p$.
One readily checks that $V \otimes V^*$ is a semisimple $G$-module, so that
$(\GL(V),G)$ is a reductive pair. 
However, as a $G(q)$-module, $V$ is isomorphic to the $G$-module
$L(1)\otimes L(1) \otimes L(1)$ which 
admits the non-simple indecomposable Weyl module of highest weight $3$ 
as a constituent, and the latter is not semisimple for $G(q)$, 
e.g., see \cite[(2D)]{Wong}.
Consequently, $V$ is not semisimple as a  $G(q)$-module. 
\end{rem}

It follows from \cite[Lem.\ 2.12(ii)(a)]{BMR} that if $H$ is $G$-cr, then $\rho(H)$ is $\rho(G)$-cr.
The following result shows that the same holds
for the saturation  
 $\rho(G)^{\sat}$ of the image of $G$ in $\GL(V)$.

\begin{cor}
\label{cor:Hsat}
Suppose that $p \ge \dim V$.
Then $\rho(G)^{\sat}$ is connected and reductive.
If $H$ is $G$-completely reducible, then 
$\rho(H)^{\sat}$ is $\rho(G)^{\sat}$-completely reducible.
\end{cor}

\begin{proof}
By Theorem \ref{thm:jantzen}, $V$ is a semisimple $G$-module. 
Lemma \ref{lem:saturation} 
then shows that 
$\rho(G)^{\sat}$ is $\GL(V)$-cr and thus 
$(\rho(G)^{\sat})^\circ$ is reductive, 
by \cite[Property 4]{serre1}.
Consider the subgroup $M$ of $\GL(V)$ generated by 
$\rho(G)$ and the closed connected subgroups
$\{\rho(u)^t \mid t \in  {\mathbb G}_a\} \cong {\mathbb G}_a$  of $\GL(V)$
for each unipotent element $u \in G$.
By definition, $M \le \rho(G)^{\sat}$.
If $M \ne \rho(G)^{\sat}$, then 
by  repeating this process with $M$ (possibly several times), we eventually generate all of 
$\rho(G)^{\sat}$ by $\rho(G)$ and closed connected subgroups of $\GL(V)$
isomorphic to ${\mathbb G}_a$. It thus 
follows from \cite[Cor.\ 2.2.7]{spr2} that $\rho(G)^{\sat}$ is connected.

Now suppose that $H$ is $G$-cr.
Theorem \ref{thm:jantzen-cr-subgps} then implies
that $V$ is a semisimple $H$-module and thus, by 
Lemma \ref{lem:saturation}, 
$V$ is also semisimple as a $\rho(H)^{\sat}$-module.
Thanks to \cite[Cor.\ 1]{serre1}, we have
$n_{\rho(G)^{\sat}}(V) \le n_{\GL(V)}(V) = \dim V -1 < p$.
It thus follows from Theorem \ref{thm:nv}(ii) 
that $\rho(H)^{\sat}$ is $\rho(G)^{\sat}$-cr, as desired.
\end{proof}

We note that a variation of  Corollary \ref{cor:Hsat} is valid for 
the general notion of saturation
replacing $\GL(V)$ with an arbitrary connected reductive group $G$.
We will  return to this in a future publication.

\begin{rem}
\label{rem:Hsat-converse}
If $p > 2 \dim V - 2$ and $V$ is non-degenerate, 
then also the converse of the final assertion of 
Corollary \ref{cor:Hsat} holds.
For, if $p > 2 \dim V - 2$, then $p \ge \dim V$. 
Thus $\rho(G)^{\sat}$ is connected and reductive, by the first part of
Corollary \ref{cor:Hsat}.
Now suppose that $\rho(H)^{\sat}$ is $\rho(G)^{\sat}$-cr.
It then follows from Lemma \ref{lem:saturation-index} and 
Theorem \ref{thm:ss-vs-cr}(i) that $V$ is a semisimple 
$\rho(H)^{\sat}$-module, and thus $V$ is a semisimple $H$-module,
by Lemma \ref{lem:saturation}.
The result now follows from Theorem \ref{thm:ss-vs-cr}(ii). 
\end{rem}

\begin{rem}
\label{rem:nv-vs-dim}
We compare the bounds $p > n(V)$ from  Theorem \ref{thm:nv}(i) and 
$p\geq \dim V$ 
from Theorem \ref{thm:jantzen-cr-subgps}.
Let $L(\mu_1),\dots,L(\mu_m)$ be the non-isomorphic simple factors
of a composition series of $V$. 

First, suppose that all $\mu_j$ are restricted.
Then $p>n(L(\mu_j))$
for each $j$, by \cite[Lem.\ 1.2]{jantzen0}.
In particular, $p> \sup\{n(L(\mu_j))\} = n(V)$, by
\cite[\S 5.2]{serre2}, so that Serre's bound applies.
In general,  $\dim V$  
is considerably larger  than $n(V)$ in this
situation. For instance, let $G$ be simple of type $E_6$ and let 
$V = L(\omega_1)$ be the  simple $G$-module of
highest weight $\omega_1$, the first fundamental dominant weight.
Here we have $n(V) = 16$, while 
$ \dim V = 27$.

Next assume that one of the $\mu_j$ is not restricted, i.e.,\  say
$\mu_j = \lambda_0 + p \lambda_1 + \dots + p^r \lambda_r$, with
restricted weights $\lambda_i$ and at least one $\lambda_i \neq 0$, ($i>0$).
According to \cite[\S 5.2]{serre2}, we find that 
$n(V) \ge n(\mu_j) = n(\lambda_0) + pn(\lambda_1)+\dots+p^rn(\lambda_r) \ge p$,
so the bound $p > n(V)$ does not apply. 

Now suppose  in addition that $p \ge \dim V$ and $H$ is $G$-cr. 
Then Theorem \ref{thm:jantzen-cr-subgps}
shows that $V$ is a semisimple $H$-module.
We can also argue as in the proof of Corollary \ref{cor:Hsat}:
Since  $p \ge \dim V$, we can saturate the image of $G$ 
in $\GL(V)$. Then, 
because $H$ is $G$-cr, it follows from  Corollary \ref{cor:Hsat} that
$\rho(H)^{\sat}$ is $\rho(G)^{\sat}$-cr.
Thanks to Lemma \ref{lem:saturation-index} and 
Theorem \ref{thm:ss-vs-cr}(i), we see  that $V$ is a semisimple 
$\rho(H)^{\sat}$-module, and thus $V$ is a semisimple $H$-module,
by Lemma \ref{lem:saturation}.
In place of Lemma \ref{lem:saturation-index} and 
Theorem \ref{thm:ss-vs-cr}(i), we can use 
Theorem \ref{thm:nv}(i) directly: 
\cite[Cor.\ 1]{serre1} implies that
$n_{\rho(G)^{\sat}}(V) \le n_{\GL(V)}(V) = \dim V -1 < p$,
so that Serre's condition is satisfied for $\rho(G)^{\sat}$
even though it is not satisfied for $G$ itself.
\end{rem}

\section{Proof of Theorem \ref{thm:sufficient-cr}}
\label{sec:proof13}
Let $H \le K$ be closed subgroups of $G$.
The normalizer of $K$ in $G$ is denoted by $N_G(K)$.
By $C_G(H) = \{g \in G \mid gxg\inverse = x \ \forall x \in H\}$ 
and $C_K(H) = C_G(H) \cap K$ we denote the centralizer of $H$ in $G$
and the centralizer of $H$ in $K$, respectively.
Analogously, 
we denote the  centralizer 
of $H$ in $\gg = \Lie G$ by 
$\cc_\gg(H) =\{y \in \gg \mid \Ad(x)y = y \ \forall x \in H\}$ 
and the centralizer of $H$ in $\kk = \Lie K$ by 
$\cc_\kk(H) = \cc_\gg(H) \cap \kk$, respectively.

Given $n\in \mathbb{N}$, we let $G$ act diagonally on $G^n$ by simultaneous conjugation:
\[
g\cdot (g_1,g_2,\ldots, g_n) = (gg_1g^{-1},gg_2g^{-1},\ldots,
gg_ng^{-1}).
\]

We require the notion of a generic tuple, \cite[Def.\ 5.4]{GIT}.
Let $G\hookrightarrow\GL_m$
be an embedding of algebraic groups.
Then $\tuple{h} = (h_1, \ldots, h_n) \in H^n$ is called a
\emph{generic tuple of $H$ for the embedding $G\hookrightarrow\GL_m$}
if the $h_i$ 
generate the associative subalgebra of $\Mat_m = \mathfrak{gl}_m = \Lie \GL_m$
spanned by $H$. We call $\tuple{h}\in H^n$ a \emph{generic tuple of $H$}
if it is a generic tuple of $H$ for some embedding $G\hookrightarrow\GL_m$.
Generic tuples exist for any embedding $G\hookrightarrow\GL_m$ if
$n$ is sufficiently large. 
The relevance to $G$-cr subgroups of this notion is as follows:
For $\tuple{h}\in H^n$ a generic tuple of $H$, 
\cite[Thm.\ 5.8]{GIT} asserts that the orbit $G \cdot \tuple{h}$ is closed in 
$G^n$ if and only if $H$ is $G$-cr.

Our first  result in this section generalizes  \cite[Rem.\ 3.31]{BMR}. 

\begin{lem}
\label{lem:sep}
Let $H \le K \le G$ be closed subgroups of $G$. 
Let $\tuple{h}$ be a generic tuple for $H$.
Then the orbit map $K \rightarrow K \cdot \tuple{h}$ is separable if and only if $H$ 
is separable in $K$.
\end{lem}

\begin{proof}
According to \cite[Lem.\ 5.5(i)]{GIT}, a generic tuple $\tuple{h}$ for $H$
satisfies the identity $C_K(H) = C_K(\tuple{h})$. 
By the same argument, we obtain
$\cc_\kk(H) = \cc_\kk(\tuple{h})$.

Let $\pi: K \rightarrow K \cdot \tuple{h}$ be the orbit map.
Then $\pi$ is separable if and 
only if $d_e \pi: \kk \rightarrow T_{\tuple{h}}(K \cdot \tuple{h})$ 
is surjective. Using 
$\dim T_{\tuple{h}}(K \cdot \tuple{h}) = \dim K\cdot \tuple{h} = \dim K - \dim C_K(\tuple{h})$ 
and $\dim \im(d_e \pi) = \dim \kk - \dim \cc_\kk(\tuple{h}) = \dim K - \dim 
\cc_\kk(\tuple{h})$, we find that the surjectivity of $d_e \pi$ is equivalent to the 
equality $\dim C_K(\tuple{h}) = \dim \cc_\kk(\tuple{h})$. By the first paragraph of 
the proof, this is equivalent to the separability of $H$ in $K$.
\end{proof}

The following generalizes part of \cite[Thm.\ 1.3]{BMRT} 
(which is the special case of Lemma \ref{lem:descend} 
when $(G,K)$ is a reductive pair and the $h_i$ lie in $K$).

\begin{lem} 
\label{lem:descend}
Let $K \le G$ be a closed subgroup. 
Let $\tuple{h}=(h_1,\dots, h_n) \in N_G(K)^n$.
Suppose that there is a decomposition $\gg = \kk \oplus \mm$ that is 
$\Ad_G(h_i)$-stable for $i=1,\dots,n$, 
and that the orbit map $\pi' : G \rightarrow G\cdot \tuple{h} \subseteq G^n$
is separable. 
Then the orbit map $\pi: K \rightarrow K\cdot \tuple{h}$ is separable. 
\end{lem}

\begin{proof}
Since $\pi':G \rightarrow G\cdot  \tuple{h}$ is separable, the
 differential $d_e\pi':\gg \rightarrow T_{\tuple{h}}(G \cdot \tuple{h})$ is surjective. 
Since $T_{\tuple{h}}(K \cdot \tuple{h}) \subseteq T_{ \tuple{h}}(G \cdot \tuple{h})$,
 any $y \in T_{\tuple{h}}(K \cdot \tuple{h})$ has a preimage $z \in \gg$ such that
 $d_e\pi'(z) = y$. 
Let $z=z_1 + z_2$ be a decomposition of $z$ in $\gg = \kk \oplus \mm$. 
Let $\mu:G^n \rightarrow G^n$ be the automorphism of varieties that sends a tuple
 $(g_1,\dots,g_n)$ to $(g_1h_1^{-1},\dots,g_nh_n^{-1})$.
Applying the differential of $\mu$ to $y$, we get 
$d_{\tuple{h}}\mu (y) = d_{\tuple{h}}\mu \circ d_e \pi' (z) = d_e
 (\mu \circ \pi') (z_1) + d_e(\mu \circ \pi')(z_2) \in \gg^n$.
Due to our assumption on the $h_i$, the map $\mu$ sends 
$K \cdot \tuple{h}$ to $K^n$, so that 
$d_{\tuple{h}}\mu (y) \in \kk^n$.
Likewise, $d_e
 (\mu \circ \pi') (z_1)\in \kk^n$.
However, 
$d_e(\mu \circ \pi')(z_2) = ((1 -\Ad h_1)(z_2),\dots,(1-\Ad h_n)(z_2)) \in \mm^n$,
according to the stability of the decomposition $\gg = \kk \oplus \mm$.  
We deduce that $d_e(\mu \circ \pi')(z_2)=0$ 
and hence $d_e(\mu \circ \pi')(z_1)=d_{\tuple{h}}\mu(y)$. 
Since $\mu$ is an automorphism, this implies that $y = d_e\pi'(z_1) = d_e \pi(z_1)$. 
We have thus shown that $d_e\pi$ is surjective, i.e., that $\pi$ is separable.
\end{proof}

Our next result generalizes \cite[Thm.\ 1.4]{BMRT} 
(which is the special case of Corollary \ref{cor:descend}
when $(G,K)$ is a reductive pair). 

\begin{cor} 
\label{cor:descend}
Let $H \le K \le G$ be closed subgroups.
Suppose that there is a decomposition $\gg = \kk \oplus \mm$ as an $H$-module
and that $H$ is separable in $G$. 
Then $H$ is separable in $K$.  
\end{cor}

\begin{proof}
Let $\tuple{h} \in G^n$ be a generic tuple of $H$.
Since $H$ is separable in $G$, the orbit map
$G \rightarrow G\cdot \tuple{h}$ is separable,
thanks to Lemma \ref{lem:sep}. 
It then follows from Lemma \ref{lem:descend} that 
$K \rightarrow K\cdot \tuple{h}$ is separable and thus that 
$H$ is separable in $K$, again by Lemma \ref{lem:sep}. 
\end{proof}

Next we give an immediate consequence of Corollary \ref{cor:descend}
and \cite[Thm.\ 1.2]{BMRT}.

\begin{cor} 
\label{cor:descend2}
Suppose that $p$ is very good for $G$.
Let $H \le K \le G$ be closed subgroups.
Suppose that there is a decomposition $\gg = \kk \oplus \mm$ as an $H$-module.
Then $H$ is separable in~$K$. 
\end{cor}

We are now in a position to prove  Theorem \ref{thm:sufficient-cr}.

\begin{proof}[Proof of Theorem \ref{thm:sufficient-cr}]
Thanks to \cite[Lem.\ 2.12(ii)(b)]{BMR}, 
we may assume that $V$ is a faithful $G$-module so that $G \le \GL(V)$.
By assumption, $H$ acts semisimply on 
$\Lie(\GL(V)) \cong V \otimes V^*$, and it is automatically
 separable in $\GL(V)$ (cf.\ \cite[Ex.\ 3.28]{BMR}). 
Since the $H$-submodule $\gg$ must have a complement 
in $\Lie(\GL(V))$, we can use Corollary \ref{cor:descend} (with ``$G=\GL(V)$'' 
and ``$K=G$'') to deduce that $H$ is separable in $G$. Moreover, as an $H$-submodule of $\Lie(\GL(V))$, 
$\gg$ is also semisimple. 
Finally, \cite[Thm.\ 3.46]{BMR} implies that $H$ is $G$-cr.
\end{proof}

We discuss some consequences of Theorem \ref{thm:sufficient-cr}.

\begin{rem}
\label{rem:linred}
It follows readily from Theorem \ref{thm:sufficient-cr} that
a linearly reductive subgroup of $G$ is $G$-cr and separable in $G$,
\cite[Lem.\ 2.6]{BMR} and \cite[Lem.\ 4.1]{rich1}.
In particular, Theorem \ref{thm:sufficient-cr} gives an alternative
proof for the first fact without any cohomology considerations,
cf.~\cite[\S 6]{rich1}.
\end{rem}

\begin{rem}
\label{rem:2}
Assume as in Theorem \ref{thm:sufficient-cr} that 
$H$ is a closed subgroup of $G$ acting semisimply on $V \otimes V^*$ 
for some faithful $G$-module $V$.
Since  $H$ is separable in $\GL(V)$ (cf.\ \cite[Ex.\ 3.28]{BMR}),
and semisimple on $\Lie(\GL(V)) \cong V \otimes V^*$, 
it follows from \cite[Thm.\ 3.46]{BMR} 
that $H$ is $\GL(V)$-cr, i.e., 
$V$ is a semisimple $H$-module 
(cf.\ \cite[(2.2.2), Prop.\ 3.2]{serre0.5}).

Moreover,
by  Theorem \ref{thm:sufficient-cr},
$H$ is $G$-cr and thus $H^\circ$ is reductive and the proof of 
Theorem \ref{thm:sufficient-cr} shows that both 
$(\GL(V),H)$ and $(G,H)$ are reductive pairs.
Note that in general $(\GL(V),G)$ need not be a reductive pair. 
\end{rem}

\begin{exmp}
Let $p = 2$, $G=\SL_2$ and let $T$ be a maximal torus of $G$.  
Let $H = N_G(T)$.  
If $\phi\colon G\ra G'$ is a non-degenerate epimorphism
(i.e., $(\ker \phi)^\circ$ is a torus), 
then $G'= \SL_2$ or $G'= \PGL_2$ and it is easily checked that 
$\phi(H)= N_{G'}(T')$, where $T':= \phi(T)$ is a maximal torus of $G'$.  
Hence $\phi(H)$ is not separable in $G'$.  
It follows from Theorem \ref{thm:sufficient-cr} 
that $H$ does not act semisimply on $V\otimes V^*$
for \emph{any} non-degenerate $G$-module $V$.
\end{exmp}

\begin{rem}
 The converse of Theorem \ref{thm:sufficient-cr} is false.  
For instance, let $p = 2 $, let $H= G= \GL_2$ 
and let $V$ be the natural module for $G$.
Then clearly $H$ is $G$-cr and separable in $G$.
But $V\otimes V^*$ is not $H$-semisimple.  
\end{rem}

\bigskip


{\bf Acknowledgments}:
The authors acknowledge the financial support of
the DFG-priority program SPP 1388 ``Representation Theory''.
Part of the research for this note was carried out while the second and fourth 
authors were visiting the Department of Mathematics at the University of York.
We would like to 
thank the members of the Department of Mathematics at the University of York
for their hospitality.
We are grateful to J-P.\ Serre for helpful comments on the material of this note.



\begin{thebibliography}{00}

\bibitem{BMR}
M.~Bate, B.~Martin, G.~R\"ohrle,
\emph{A geometric approach to complete reducibility},
Invent.\ Math. \textbf{161}, no. 1 (2005), 177--218.



\bibitem{BMRT}
M.~Bate, B.~Martin, G.~R\"ohrle, R.~Tange,
\emph{Complete reducibility and separability},
Trans.\ Amer.\ Math.\ Soc., \textbf{362} (2010), no. 8, 4283--4311.

\bibitem{GIT}
\bysame, 
\emph{Closed orbits and uniform $S$-instability in geometric invariant theory}, (2009),
{\tt arXiv:0904.4853v3 [math.AG]}








\bibitem{herpel}
S.~Herpel,
\emph{On the smoothness of centralizers in reductive groups},
preprint (2010), {\tt arXiv:1009.0354v3 [math.GR]}





\bibitem{jantzen0}
J.~C. Jantzen,
\emph{Low-dimensional representations of reductive groups are semisimple}.
\newblock In {\em Algebraic groups and {L}ie groups}, volume~9 of {\em Austral.
  Math. Soc. Lect. Ser.}, pages 255--266. Cambridge Univ. Press, Cambridge,
  1997.



\bibitem{liebeckseitz0}
M.W.~Liebeck, G.M.~Seitz,
\emph{Reductive subgroups of exceptional algebraic groups}.
Mem.\ Amer.\ Math.\ Soc.\ no.\ \textbf{580} (1996).







\bibitem{mcninch2}
G.~McNinch,
\emph{Dimensional criteria for semisimplicity of representations},
Proc.\ London Math.\ Soc.\ (3) \textbf{76} (1998), no.\ 1, 95--149.





\bibitem{rich2} 
R.W.~Richardson,
\emph{Conjugacy classes in Lie algebras and algebraic groups},
Ann.\  Math.\  \textbf{86}, (1967), 1--15.


\bibitem{rich1}
\bysame, 
\emph{On orbits of algebraic groups and Lie groups},
Bull.\  Austral.\  Math.\  Soc.\  \textbf{25} (1982), no.\  1, 1--28.





\bibitem{serre0}
J-P. Serre,
\emph{Sur la semi-simplicit\'e des produits tensoriels de repr\'esentations de groupes},
Invent. Math. \textbf{116} (1994), no. 1-3, 513--530. 


\bibitem{serre0.5}
\bysame, 
\emph{Semisimplicity and tensor products of group representations:
converse theorems}. With an appendix by Walter Feit,
J. Algebra \textbf{194} (1997), no.\ 2, 496--520.



\bibitem{serre1}
\bysame, 
\emph{The notion of complete reducibility in group theory},
Moursund Lectures, Part II, University of Oregon, 1998,\
{\tt arXiv:math/0305257v1 [math.GR]}.

\bibitem{serre2}
\bysame, 
\emph{Compl\`ete r\'eductibilit\'e},
S\'eminaire Bourbaki, 56\`eme ann\'ee, 2003--2004, n$^{\rm o}$ 932.

\bibitem{slodowy}
P.~Slodowy,
\emph{Two notes on a finiteness problem in the representation theory
of finite groups},
Austral. Math. Soc. Lect. Ser., \textbf{9},
Algebraic groups and Lie groups, 331--348,
Cambridge Univ. Press, Cambridge, 1997.

\bibitem{spr2}
T.A. ~Springer, \emph{Linear algebraic groups},
Second edition. Progress in Mathematics, 9. Birkh\"auser Boston, Inc.,
Boston, MA, 1998.



\bibitem{Wong}
W. J. Wong, 
\emph{Irreducible modular representations of finite Chevalley groups},
J. Algebra  \textbf{20}  (1972), 355--367. 

\end{thebibliography}
\end{document}